\documentclass[12pt]{amsart}
\usepackage{amsthm}
\usepackage[utf8]{inputenc}
\usepackage{appendix}
\usepackage[top= 2.5cm, bottom=2.5cm, left=3cm, right=3cm]{geometry}
\usepackage{cite}
\usepackage{amssymb,amsmath,amsthm}
\usepackage{hyperref}
\usepackage{mathtools}
\usepackage{color, xcolor}
\usepackage[all]{xy}
\usepackage{tikz}
\usepackage{pgfplots}
\pgfplotsset{compat=1.16} 
\usetikzlibrary{patterns}
\usepgfplotslibrary{fillbetween}
\usepackage{caption}

\usepackage{url}

\newcommand{\amsprimary}[1]{{\footnotesize\noindent AMS 2010 \textit{Mathematics subject
classification:} Primary #1\vspace{1pc}}}
\newcommand{\keywordsnames}[1]{{\footnotesize\noindent\textit{Key words:} #1\vspace{1pc}}}

\newtheorem{teo}{Theorem}
\newtheorem{prop}{Proposition}[section]

\newtheorem{lemma}{Lemma}

\theoremstyle{definition}
\newtheorem{dfn}{Definition}[section]

\newtheorem{rmk}{{\bf Remark}}
\newtheorem{example}{{\bf Example}}
\newcommand{\bpi}{{\mathcal{B}_{\pi}}}
\newcommand{\dpi}{{d_{\pi}}}
\newcommand{\Hpi}[1]{{H^{#1}_{\pi}}}
\newcommand{\Cas}[1]{{\rm Cas}_{\pi}({#1})} 
\newcommand{\Vpi}[1]{\mathcal{V}^{#1}}
\newcommand{\HdR}[1]{{H^{#1}_{dR}}}
\newcommand{\FRb}[1]{({#1}, \{, \}_{\mu})}
\newcommand{\FR}{{\bf JP\ }}

\newcommand{\cproc}[2]{\wedge_{{#1}{#2}}}

\title[Poisson Cohomology of singular fibrations]{Poisson cohomology of singular fibrations in dimension $4$}
\author{N. Bárcenas}

\address[N. BÁRCENAS]{Centro de Ciencias Matemáticas\\
	  Universidad Nacional Autónoma de México\\
	  Morelia, Mich., México}
	\email{barcenas@matmor.unam.mx}
	
	\author{J. Torres Orozco}
	\address[J. TORRES OROZCO]{Facultad de Ciencias\\
	  Universidad Nacional Autónoma de México\\
	  Ciudad de México, México}
	\email{jonatan.tto@gmail.com}

\date{}

\begin{document}
\maketitle

\begin{abstract}
It is known that there exist singular Poisson structures in $4$-manifolds, whose symplectic foliation is given by singular fibrations over surfaces. In this work, we describe the effect of the monodromy of the fibration on the Poisson cohomology groups. 
\end{abstract}

{\keywordsnames {Poisson cohomology, generic maps, Thom class, monodromy representation.}}
	
	{\amsprimary {53D17, 55N35, 55R55, 57K20, 57R45.}}

\section{Introduction}
%
%
%
%

A Poisson structure on a smoooth manifold $M$ is a Lie algebra structure on $C^{\infty}(M)$, the space of smooth real valued smooth functions on $M$, whose bracket $\{ , \}$ satisfies a Leibniz rule. The Lie bracket induces an anchor map that additionally endows $T^*M$ with a Lie algebroid structure. This algebraic perspective enriches the understanding of the geometry and topology of the underlying manifold, from both local and global point of view. 

A manifold endowed with such a structure can be described as a foliated space, whose leaves inherit a symplectic form. This is known as the symplectic foliation. Poisson structures yield  a  more  general notion of  a  symplectic  structure,  since they can be defined on odd dimensional  smooth  manifolds. The  tools  needed  to  understand a  manifold that  admits  a  Poisson structure include  aspects of singularity  theory. In  particular, the  classification  of  singularities  is  needed, as they provide "paradigmatic", "canonical" singular Poisson  structures.  
\newline

Poisson structures are closely related to near-symplectic structures and singular fibrations. S. Donaldson establishes a correspondence between Lefschetz pencils and symplectic manifolds of dimension $4$ \cite{Donaldson99}. Lefschetz pencils are applications on the $2$-sphere that have a finite number of isolated singular points, where the differential has rank equals zero. D. Auroux, S. Donaldson, and L. Katzarkov \cite{ADK05} introduced broken Lefschetz fibrations (BLF for short), which have an additional type of singularity, a $1$-dimensional submanifold formed by indefinite folds. They showed that they are reciprocal to near-symplectic structures, which are closed $2$-forms that are non-degenerate, except on a collection of circles where they vanish. 
\newline

In a sequel of works, it was shown that Poisson structures adapted to singular fibrations can be constructed. The singular symplectic leaves matches with those of the inherent fibration. In \cite{GSSV15}, the authors proved the existence of rank 2 Poisson structures in BLF's. Namely, it is known that a BLF always exists on any closed oriented smooth manifold of dimension $4$. As extension of the previous one, in \cite{SST16}, Poisson structures were constructed  on  4-manifolds  which  admit  a  wrinkled  fibration.  A similar approach was given in dimension $6$, for maps that generalize wrinkled fibrations, and near-symplectic structures were studied in \cite{SSTV19}. In each case, a similar procedure was carried out, by prescribing the singularities, given by the corresponding fibration and then perform the local calculations of the Poisson bivectors as well as the symplectic forms of the symplectic folation.
\newline

One natural way to study global properties is by computing the cohomology. In the case of a symplectic estructure, it means to compute the de Rham cohomology; for a Poisson structure, it means the cohomology of $T^*M$ as Lie algebroid. The latter is known as {\em Poisson cohomology}. The Poisson cohomology groups, denoted by $\Hpi{*}(M)$, provide information about the Poisson structure in the sense that:
\begin{eqnarray*}
\Hpi{0}(M)&=&\{f\in C^{\infty}(M)|\ \{f, \cdot\}=0\}, \mbox{\ referred to as the space of Casimirs,}\\
\Hpi{1}(M)&=&\frac{\{\mbox{Poisson vector fields}\}}{\{\mbox{Hamiltonian vector fields}\}},\\
\Hpi{2}(M)&=&\frac{\{\mbox{infinitesimal deformations of the Poisson structure}\}}{\{\mbox{trivial deformations}\}},\\
\Hpi{3}(M)&=&\{\mbox{Obstructions to deformations of Poisson structures}\}.
\end{eqnarray*}
For the higher cohomology groups, neither their algebraic nor their geometric interpretations are known. 
\medskip

The computation of Poisson cohomology requires much effort, in comparison to de Rham cohomology. One way is by a direct calculation of the corresponding differentials, or via spectral sequences in terms of filtrations. See for instance \cite{Monnier02, Monnier05, Pichereau06, BatVer20, Lanius21}.
\newline

The contribution of this work is a description of some Poisson cohomology classes for a Poisson structure whose symplectic foliation is prescribed by a singular fibration on $4$-manifolds, those classes that enclose the topological information of the general fiber around a singular point. The mechanism is the combinated use of the Thom isomorphism (for de Rham cohomology), together with the monodromy representation i.e, the homomorphism of groups:
\[
\pi_1(\Sigma\setminus {\rm Sing}(f))\to \mathcal{M}_g
\]
associated to a singular fibration $f:M\to \Sigma$ over a surface $\Sigma$, with singular locus ${\rm Sing}(f)$. Here $\mathcal{M}_g$ is the mapping class group of $\Sigma_g$, the general fiber around singularities. We provide a description of the action of the monodromy of the fibration on the symplectic foliation, which is reflected in the first and third Poisson cohomology groups. 
\newline

\subsubsection*{Outline} Section 2 is about the background on Poisson Geometry and Singularity Theory of generic maps and their deformations. In Section 3, we review the construction of Poisson structures of rank 2, from two prescribed Casimirs. In our setting, we consider singular fibrations over 2-dimensional manifolds. Later, in Section 4, we obtain the image of the Thom class of any symplectic leaf in a Poisson manifold. We also exhibit its relation with its algebroid version. Finally, in Section 5, we describe Poisson cohomology classes and describe the effect of the monodromy of the underlying fibration.

\section{Preliminaries}

This section contains the relevant basics on Poisson Geometry and Singularity Theory to follow this work. It has been adapted to the purpose of this paper. For a most detailed explanation we refer the reader to \cite{Vaisman94} and  \cite{GolGui73}.
\subsection{Poisson manifolds}

\medskip
A {\em Poisson manifold} is a pair $(M, \{,\})$ of a smooth manifold of dimension $n$ and a bilinear operation $\{\cdot , \cdot \}$ on $C^\infty(M)$, the space of real valued smooth functions on $M$, with the properties:
 \begin{enumerate}
\item[(i)] $\left( C^\infty(M) , \{\cdot , \cdot \}\right)$ is a Lie algebra.
\item[(ii)] The bracket $\{\cdot , \cdot \}$ is a derivation in each factor, that is,
\begin{equation*}
\{gh, k\}=g\{h, k\}+ h\{g, k\}
\end{equation*}
for any $g,h,k\in  C^\infty(M)$.
\end{enumerate}
\medskip

Given a function $h\in C^\infty(M)$  the {\em Hamiltonian vector field} $X_h$ is defined as $X_h(\cdot )=\{\cdot , h \}$.
\newline

The most basic example of a non-trivial Poisson manifold is a symplectic manifold $(M,\omega)$. The bracket on $M$ is given by
\begin{equation}
\label{PoisSym}
\{g,h\}=\omega(X_g,X_h).
\end{equation}
The Jacobi identity for the bracket follows from the property of $\omega$ of being closed.  
\newline

The bracket of a Poisson manifold can be regarded as a contravariant antisymmetric 2-tensor $\pi$ on $M$:
\begin{equation}\label{eq:bracket-bivector}
\{g,h\}=\pi(dg,dh).
\end{equation}
We will refer to $\pi$  as the {\em Poisson tensor}, or {\em Poisson bivector}, and it will be used to mean a Posson structure on $M$. In local coordinates $(x^1, \dots , x^n)$ we have the expression:
\begin{equation}
\label{eqn: localPoisson}
\pi(x)=\frac{1}{2}\sum_{i,j=1}^n\pi^{ij}(x)\frac{\partial}{\partial x^i}\wedge \frac{\partial}{\partial x^j}, \quad \pi^{ij}(x)=\{x^i,x^j\}.
\end{equation}

The Jacobi identity for the bracket can be expressed as $[\pi,\pi]_N=0$, where $[\cdot,\cdot]_N$ is the Schouten-Nijenhuis bracket of multivector fields. 

\medskip

\begin{dfn}
\label{D:Casimir}
Let $(M, \{, \})$ be a Poisson manifold. A function $h\in C^\infty(M)$ is called a {\em Casimir} if $\{h,g\}=0$ for
every $g\in C^\infty(M)$. The space of all Casimirs will be denoted by $\Cas{M}$.
\end{dfn}
\medskip

\begin{dfn}
Let $\pi$ be a bivector (Poisson or not). Then we can associate to it a bundle map $\bpi:T^*M\to TM$ to $\pi$ defined by its action on covectors by the fiberwise rule 	
\[
\bpi(\alpha_q)(\cdot)=\pi_q(\cdot ,\alpha_q),
\]
at a point $q\in M$, for $\alpha_q\in T_q^*M$. This bundle map is called the {\em anchor map} of $\pi$.
\end{dfn}
Observe that, if $\pi$ is a Poisson bivector, the Hamiltonian vector field of any smooth function $h$ on $M$ is $X_h=\bpi(dh)$. Additionally, we may rewrite the bivector (\ref{eqn: localPoisson}) into the form:
\begin{equation}
\label{eqn: localPoisson2}
\pi(x)=\frac{1}{2}\sum_{i,j=1}^n\bpi(dx^i)\wedge \frac{\partial}{\partial x^j},
\end{equation}
and by duality, for $\{(dx^1)_q, \dots, (dx^n)_q\}$, the canonical basis of $T_q^*M$, 
\[
\bpi((dx^i)_q)=\sum_{j=1}^n \pi^{ij}(q)\frac{\partial}{\partial x^j}.
\]


 The {\em rank} of a $\pi$  at a point $q\in M$ is defined to be the rank of $\bpi:T^*_qM\to T_qM$. At the point $q$, the image of $\bpi$ is a subspace $D_q\subset T_qM$, and the collection of these subspaces as $q$ varies on $M$ defines the so-called {\em symplectic foliation} of $\pi$. Then the image of $\bpi$ at $q$ is a leaf of the symplectic foliation, whose dimension is the rank of $\bpi$ at $q$. Observe that its rank is even and coincides with the rank of $\pi$, but it may no be constant, so the symplectic foliation becomes {\em singular}. The rank of $\pi$ at $q\in M$ is called the {\em rank of the Poisson structure at $q$}.
\newline

The elements of the symplectic foliation are referred to as {\em symplectic leaves}, since they admit a unique symplectic form. Indeed, if $u_q$ and  $v_q$ are vectors of the symplectic leaf $\Sigma_q$, such symplectic form $\omega_q$ is given by the natural pairing $\langle\ ,\ \rangle$ between $T_q^M$ and $T_q^*M$:
\[
\omega_q(u_q, v_q):=\left\langle\pi, \alpha_q\wedge\beta_q \right\rangle
\]
where $\alpha, \beta\in T_q^*M$ such that
\begin{equation}
\label{eqn: AuxSymp}
\bpi(\alpha_q)=u_q,
 \quad \bpi(\beta_q)=v_q,
\end{equation}
or equivalently, 
\[
\omega(u_q, v_q)_q=\pi(\bpi^{-1}(u_q), \bpi^{-1}(v_q)).
\]

If the inherent Poisson structure have a singularity at $q\in M$, then the matrix $\pi^{ij}(q)$ is not of maximum rank and both equations produce overdetermined systems. 
\medskip

\subsection{Poisson cohomology}

Let $(M, \pi)$ be a Poisson manifold. Denote by $\Vpi{p}M$ the space of smoth $p$-vector fields on $M$. Consider $\dpi: \Vpi{\bullet}M\to \Vpi{\bullet+1}M$ given by $\dpi(A):=[\pi, A]_{N}$. Since $[\pi, \pi]_{N}=0$, then $\dpi$ is a differential operator for the complex
\[
\cdots \xrightarrow[]{\dpi}\Vpi{p-1}\xrightarrow[]{\dpi} \Vpi{p}M\xrightarrow[]{\dpi} \Vpi{p+1}M\xrightarrow[]{\dpi}\cdots,
\]
and the quotient groups
\[
\Hpi{p}(M)=\frac{\ker(\dpi: \Vpi{p}M\to \Vpi{p+1}M) }{{\rm Im}(\dpi: \Vpi{p-1}M\to \Vpi{p}M)}
\]
are called the {\em Poisson cohomology groups} of $(M, \pi)$.
\newline


The anchor map $\bpi: T^*M\to TM$ can be extended to the $C^{\infty}$-linear homomorphism:
\begin{equation}
\bpi: \Omega^p(M)\to \Vpi{p}M, \quad \bpi(\eta)(\alpha_1, \dots, \alpha_p)=(-1)^p\eta(\bpi(\alpha_1), \dots, \bpi(\alpha_p)).
\end{equation}
It preserves the wedge product of forms, and the relation:
\[
\bpi(d\eta)=-[\pi, \bpi(\eta)]_N
\]
holds, which induces a homomorphism of graded Lie algebras from the de Rham cohomology and the Poisson cohomology:
\[
\bpi:\bigoplus_p \HdR{p}(M)\to \bigoplus_p \Hpi{p}(M), \quad \ [\eta] \mapsto [\bpi(\eta)].
\]
It is known as the {\em Lichnerowicz homomorphism}. In general, it is not an isomorphism. 


\subsection{Generic maps on $4$--manifolds over closed surfaces}

In this section, $f\colon X \rightarrow \Sigma$ will be a smooth map between two closed smooth manifolds with $\dim(X)=4$ and $\dim(\Sigma)=2$, with differential map $df\colon TX \rightarrow T\Sigma$. The space of smooth maps from $X$ to $\Sigma$ will be denoted by $C^{\infty}(X, \Sigma)$. 
\newline

A point $p\in X$ is {\em regular} if the rank of $df_{p}$ has dimension $2$; otherwise, the rank must be $0$ or $1$, and the point $p\in M$ is a {\em singularity of $f$}, while the set
\begin{equation*}
{\rm Sing}(f)= \lbrace p\in M \mid {\rm Rank}(df_p)= 1 \rbrace
\end{equation*}
is named the {\it singularity set} or {\em singular locus} of $f$. Therefore, around singular point $p$ there are local coordinates such that $f$ is given by $(t, x, y, z)\mapsto (t, \psi(t, x, y, z))$ for some smooth function $\psi$ in a nearby of $p$. 
\newline

Two maps $f, \tilde{f}\in C^{\infty}(X, Y)$, between two smooth manifolds $X$ and $Y$ are said to be {\em equivalent} if there exists diffeomorphisms $\psi: X\to X$ and $\varphi: Y\to Y$ such that
\[
\varphi\circ f=\tilde{f}\circ \psi.
\]
A property {\bf P} of smooth mappings $f:X\to \Sigma$ is {\em generic} if:
\begin{itemize}

\item The set $W_{\mathbf{P}}=\{f\in C^{\infty}(X, Y)| \mbox{$f$ satisfies $\mathbf{P}$}\}$ is open and dense in $C^{\infty}(X, Y)$, and

\item if  an application $f$ is in $W_{\mathbf{P}}$, then any equivalent application to $f$ also belongs to $W_{\mathbf{P}}$.

\end{itemize}

A function satisfying a generic property is called a {\em generic function.} A singularity described (locally) by a generic function is a {\em generic singularity}.
\medskip

In general, for a generic map, the set of points of corank $r$, $r=0, 1, \dots, {\rm dim}(X)$, is a submanifold of $X$, and the restriction of $f$ at the singular locus of those points of corank $r$ gives a smooth map between manifolds that can also have generic singularities (see Theorem 5.4 p.61 \cite{GolGui73}).
\newline

The points in the set ${\rm Sing}(f)$ satisfying $T_p{\rm Sing}(f) \oplus \ker(df_p) = T_p M$ are called {\it fold} singularities of $f$. In particular, a submersion with folds i.e, a submersion outside the set of folds, restricts to an immersion on its fold locus (see Lemma 4.3 p.87 \cite{GolGui73}). Folds are locally modelled by 
\begin{equation}
\label{eqn: folds}
\mathbb{R}^4 \rightarrow \mathbb{R}^2,  \quad \left( t, x, y, z \right) \rightarrow \left( t, \pm x^2 \pm y^2 \pm z^2\right).
\end{equation}
{\it Cusps} are points $p$ belonging to ${\rm Sing}(f)$ such that $T_p {\rm Sing}(f) = \ker(df_p)$. In dimension $4$, they are parametrized by: 
\begin{equation}
\label{eqn: cusps}
\mathbb{R}^4 \rightarrow \mathbb{R}^2 \quad \quad \left( t, x, y, z \right) \rightarrow \left(t, x^3 + t\cdot x \pm  y^2 \pm z^2\right).
\end{equation}


%
A classical result due to Whitney says that generic maps from any $n$--dimensional manifold to a 2--dimensional base only have folds and cusps.  Then a generic map over a $2$-dimensional  manifold admits one of the of the two forms (\ref{eqn: folds}) or (\ref{eqn: cusps}) around its singularities.
\medskip

\begin{dfn}
\label{D:BLF}
A {\it broken Lefschetz fibration} (BLF) is a surjective smooth map $f\colon X \rightarrow \Sigma$ that is a submersion outside the singularity set, where the only allowed singularities are of the following type:
\begin{enumerate}
\item {\em Lefschetz singularities}:  finitely many points  $\{ p_1, \dots , p_k\} \subset X,$  which are locally modelled by
real charts:
\[
\mathbb{R}^4\to \mathbb{R}^2, \quad\quad (t, x, y, z)\mapsto (t^2-x^2+y^2-z^2, 2tx+2yz).
\]

\item {\it Indefinite fold singularities}, or also called {\it
broken singularities}, contained in the smooth embedded 1-dimensional
submanifold $\Gamma \subset X \setminus  \{ p_1, \dots , p_k\} $,
which are locally modelled by the real charts
\[ 
{\mathbb{R}}^{4} \rightarrow{\mathbb{R}}^{2} ,  \quad \quad  (t,x, y, z) \mapsto (t, - x^{2} + y^{2} + z^{2}).
\]
The curve $\Gamma$ is called a {\em singular circle.}
\end{enumerate}
If $f:X\to \Sigma $ has no broken singularities, it is called a {\em Lefschetz fibration}. 
\end{dfn}

The singular circles of folds could intersect each other or the Lefschetz points could lie between the singular circles. Nevertheless, along this work we are assuming that a BLF has only one singular circle $\Gamma$, and a finite set of Lefschetz singularities outside $\Gamma$. This kind of BLF are known as {\em simplified broken Lefschetz fibrations}.
\newline

In a nearby of a singular fiber at a Lefschetz singularity, the fibers are diffeomorphic to closed surfaces attached to the singularity. As a regular fiber approaches a Lefschetz singularity, the curve (up to homotopy) shrinks to a point. Such curves are called {\em vanishing cycles}. For a point in the singular circle, the genus of the fiber drops down by 1. In Section \ref{sec: PclassMon} we will return to the topology of a BLF. 
\newline

\begin{dfn}
A map $f:X\to Y$ between smooth manifolds is said to be stable if any nearby map $\tilde{f} \in C^{\infty}(X,\Sigma)$ (in the Whitney topology) is equivalent to $f$. 

 \end{dfn}


\example Given Morse function $\psi:\mathbb{R}^3\to \mathbb{R}$ whose critical points have only index 1 or 2, then the mapping $(t, x, y, z)\to (t,\psi(x, y, z) )$ is a BLF whose singular locus corresponds to the critical points of $\psi$. BLF's are not stable maps. Furthermore, if $M$ is a closed, oriented manifold and $\psi: M\to \mathbb{R}$ a Morse function. Then $f: M\times S^1\to \mathbb{R}^2$, $(x, \theta)\mapsto (\psi(x), \theta)$ defines a BLF. 
\newline

Following the work by Y. Lekili, there is class of stable maps that naturally appear when deforming BLF's around a Lefschetz singularity \cite{Lekili09}. Those fibrations exist in every closed oriented manifold of dimension $4$. 

 \begin{dfn} A {\em wrinkled fibration} on a closed $4$--manifold $X$ is a smooth map $f$ to a closed surface which is a BLF when restricted to $X\setminus C$, where $C$ is a finite set such that around each point in $C$, $f$ has cusp singularities. It is called a  {\it purely wrinkled} fibration if it has no isolated Lefschetz-type singularities.
\end{dfn}
\medskip

Part of the work by Lekili, consisted in describing a set of moves that give all the possible one-parameter deformations of broken and wrinkled fibrations, up to isotopy. Roughly speaking, it is possible to eliminate a Lefschtez type singularity on a BLF by introducing a wrinkled fibration structure; as well as there exists a mechanism to smoothing out the cusp singularity by introducing a Lefschtez singularity. This also shows the stability of wrinkled fibrations. The Lekili's moves are real $1$-parameter applications, locally given by:

\begin{enumerate}
\item{\bf Birth}
\begin{equation*}
b_s(x, y, z, t)=(t, x^3-3x(t^2-s)+y^2-z^2).
\end{equation*}

\item{\bf Merging}
\begin{equation*}
m_s(x, y, z, t)=(t, x^3-3x(s-t^2)+y^2-z^2).
\end{equation*}

\item{\bf Flipping}
\begin{equation*}
f_s(x, y, z, t)=(t, x^4-x^2s+xt+y^2-z^2).
\end{equation*}

\item{\bf Wrinkling}
\begin{equation*}
w_s(x, y, z, t)=(t^2-x^2+y^2-z^2+st, 2tx+2yz).
\end{equation*}
\end{enumerate}
\medskip

Any generic deformation of a surjective map $f:X\to \Sigma$, around a critical point of rank $1$, is given by one of the first 3 moves. The wrinkling move is a deformation of a Lefschetz singularity, where its differential map vanishes, which changes the Lefschetz singularity into a singularity into $3$ cusps. The existence of points where $df_p$ vanishes is not a generic property. 
\medskip

\section*{Acknowledgments}

We thank Pablo Suárez-Serrato and Camilo Arias  for useful discussions and their suggestions on the work. The first author thanks DGAPA Project IN100221 and CONACYT through grant CF 217392. 

\section{Poisson structures with prescribed singularities}

In this section, $M$ will denote an oriented smooth manifold of dimension $n$, with $\mu$ an orientation, and $F_1, \dots, F_{n-2}$ will be fixed functions in $C^\infty(M)$. Consider a bivector on $M$ defined by the relation
\begin{equation}
\label{FRform}
\{g, h\}\mu=\pi(dg, dh)\mu=k\, dg\wedge dh\wedge dF_1\wedge\dots \wedge dF_{n-2}
\end{equation}
for smooth functions $g$ and $h$ on $M$, and $k$ is a non-vanishing smooth function. Note that the skew-symmetric matrix $\pi^{ij}$ annihilates $dF_i$, $i=1, \dots, n-2$, and it has rank at most two. Also note that $\mu$ can be chosen to be degenerate.
\newline

Set $F:=(F_1, \dots ,F_{n-2}):M\to  \mathbb{R}^{n-2}$. The symplectic foliation of a bivector $\pi$ given by (\ref{FRform}) is integrable and its regular leaves are given by $F^{-1}(y)$ with $y\in \mathbb{R}^{n-2}$ a regular value, and by $F^{-1}(y)\setminus \{\mbox{Critical Points
of $F$}\}$ for $y$ a critical value; the singular leaves correspond to critical points of $F$. It was proved in \cite{GSSV15}. that such bivector $\pi$ is Poisson. In fact, $\tilde{\pi}=k\cdot \pi$ is a Poisson bivector for any non-vanishing function $k\in C^{\infty}(M)$.

%
%
%
%

The functions $F_1, \dots, F_{n-2}$ are Casimirs for the corresponding bracket, and the singularities of $F$ will determine the singular structure of $\pi$. That is, formula (\ref{FRform}) provides a mechanism for constructing singular Poisson manifolds, from prescribed singularities. This was explored in \cite{GSSV15} for a BLF; in \cite{SST16} and \cite{SSTV19} for wrinkled fibrations in dimension $4$ and $6$; with an slight modification, in \cite{ESSTV19} it was adapted to Bott-Morse foliations in dimension $3$. The objective of this section is to give a more general form of such construction, but adapted to $4$--manifolds. 

\rmk The freedom on the choice of the function $k$ follows since we are considering Poisson structures of rank $2$. In this case, a leafwise $2$-form is closed if and only if the Jacobi identity holds.

\begin{dfn}
Given $F_1, \dots, F_{n-2}\in C^{\infty}(M)$, a {\em Jacobian Poisson structure} of dimension $n$ on an oriented manifold $M$, with orientation $\mu$ is a Poisson bracket $\{, \}_{\mu}$ on $M$ given by the formula (\ref{FRform}). For short, we simply write {\em \FR structure}, and the respective \FR manifold will be denoted by $\FRb{M}$.
\end{dfn} 


Hereinafter, we will be focused on \FR structures on smooth manifolds of dimension $4$. Then they are determined by two functions $F, G\in C^{\infty}(M)$. 
\newline

Each element of $S_4$, the symmetric group on four letters, is a bijection $\sigma: \{1, 2, 3, 4\}\to \{1, 2, 3, 4\}$. For fixed indices $i, j$, consider the assignment $(i, j)\mapsto \sigma= (ijrs)\in S_4$, where $(ijrs)$ denotes the permutation  given in the cyclic notation. Observe that, once $i$ and $j$ are fixed, there exist a unique pair $(r, s)$ such that $\sigma$ is an even permutation.

\begin{dfn}
For two vector fields $X$ and $Y$ in $\mathbb{R}^4$, and two fixed indices $i, j=1, \dots, 4$, we define the skew-symmetric bilinear operator $\cproc{i}{j}$ given by:
\begin{equation}
X\cproc{i}{j}Y:=dx^r\wedge dx^s(X, Y)
\end{equation}
where $(r, s)$ is the unique pair such that $(ijrs)$ is an even permutation.
\end{dfn}
\medskip

On the other hand, formula (\ref{FRform}) implies that: 
\begin{equation*}
\pi^{ij}=k\cdot \sum_{r, s=1}^4\epsilon_{ijrs}\frac{\partial F}{\partial x^r}\frac{\partial G}{\partial x^s}=k\cdot \left( \frac{\partial F}{\partial x^{\bar{r}}} \frac{\partial G}{\partial x^{\bar{s}}}- \frac{\partial F}{\partial x^{\bar{s}}}\frac{\partial G}{\partial x^{\bar{r}}}\right)
\end{equation*}
where $\epsilon_{ijrs}$ denotes the Levi-Civita symbol in dimension $4$, and $(\bar{r}, \bar{s})$ is the unique pair such that $(ij\bar{r}\bar{s})$ is an even permutation. Therefore we immediately obtain:
\medskip

\begin{prop}
\label{prop: FRlocal}
The local expression of the Poisson bivector of a \FR manifold of dimension $4$ with Casimirs $F$ and $G$, $\FRb{M}$, around a point $q\in M$, is given by:
\begin{equation}
\label{eqn: PoissBiv}
\pi^{ij}=k\cdot  \nabla F\cproc{i}{j}\nabla G.
\end{equation}
\end{prop}
\medskip

\begin{prop}
Let $U$ be a neighborhood of a singular point $p$ of $f$. Then the symplectic form $\omega_q$ of the symplectic leaf at any $q\in U\setminus\{p\}$ is given by:
\[
\omega_q= \langle \alpha_q, v_q\rangle =\langle \beta_q, u_q\rangle 
\]
where $u_q, v_q$ are tangent to the symplectic leaf $S$ and $\langle, \rangle$ denotes the natural pairing between differential forms and vector fields.
\end{prop}
\begin{proof}
In order to compute the symplectic forms we need to solve equations (\ref{eqn: AuxSymp}). Their solutions depend merely on the first partial derivatives of $F$ and $G$. In fact, since ${\rm Ann}\ TS={\rm Ker} (\bpi)$, then the tangent vectors $u_q, v_q$ can be found by seeking for vectors annihilated simultaneously by $dF$ and $dG$. Now, $\alpha_q$ and $\beta_q$ are solutions to the overdetermined system  (\ref{eqn: AuxSymp}) that satisfy: $\bpi(\alpha_q)=u_q, \bpi(\beta_q)=v_q$.\end{proof}

In the sequel we will assume that the volume form in local coordinates $(t, x, y, z)$ around a point $q\in M$ has the form:
\[
(\mu)_q=\frac{1}{k(q)} (dt\wedge dx\wedge dy\wedge dz)_q
\]
for some non-vanishing $k\in C^{\infty}(M)$. 
\newline

Since generic maps can only have folds and cusps, using Proposition \ref{prop: FRlocal} one has:
\begin{teo}
Given a (singular) generic map $F: M\to \Sigma$ from an oriented $4$--manifold on a closed surface, with $\mu$ the volume form on $M$, there exists a singular \FR structure on $M$, whose singularities coincide with the generic singularities of $F$. Locally $F=(t, \psi(t, x, y, z))$ and the Poisson bivector around a generic singularity takes the form
\begin{equation}
\label{eqn: GenPoissBiv}
\pi(p)=k(p)\left[\frac{\partial \psi}{\partial y} \frac{\partial }{\partial x}\wedge \frac{\partial }{\partial z}+\frac{\partial \psi}{\partial z} \frac{\partial }{\partial x}\wedge \frac{\partial }{\partial y}-
\frac{\partial \psi}{\partial x} \frac{\partial }{\partial y}\wedge \frac{\partial }{\partial z}\right].
\end{equation}
for some non-vanishing smooth function $k$ on $M$. The regular leaves of the symplectic foliation are the non-singular fibres of $F$. The bivector is singular at the critical points of $\psi$.
\end{teo}
\medskip

Lefschetz type singularities are not generic. Nevertheless, a Poisson structure can be obtained in the same way, and it has the form (\ref{eqn: PoissBiv}). Moreover, since Lefschetz and Broken Lefschetz fibrations exist on every $4$-manifold $M$, then a singular \FR structure always exists on any closed oriented smooth manifold of dimension $4$. This was shown in \cite{GSSV15}. 
\newline

\rmk These kind of strtuctures are quite similar to those considered by A. Picherau in \cite{Pichereau06}, where computation of the Poisson cohomology groups were given for structures in $\mathbb{R}^3$ given by:
\[
\pi(p)=k(p)\left[\frac{\partial \varphi}{\partial y} \frac{\partial }{\partial x}\wedge \frac{\partial }{\partial z}+\frac{\partial \varphi}{\partial z} \frac{\partial }{\partial x}\wedge \frac{\partial }{\partial y}+
\frac{\partial \varphi}{\partial x} \frac{\partial }{\partial y}\wedge \frac{\partial }{\partial z}\right].
\]
for $\varphi: \mathbb{R}^3\to \mathbb{R}$ a weight homogeneous polynomial with an isolated singularity. 

\subsubsection*{From local to global}

The bivectors given in the previous are local, but they can be extended to define a global Poisson structure, by means of cut-off functions. The mechanism follows merely topological arguments, we will explain it in a more general form, by considering two types of singularities: isolated points or (singular) circles; and maps $F:X\to \Sigma$ (generic or not) with those singularities.
\medskip

Let $F:X\to \Sigma$ be a surjective map, which is a submersion outside its singular locus ${\rm Sing}(f)$. It is allowed to be generic or not. Let $U_{\Gamma}$ be the union of tubular neighborhoods of circle-type singularities (e.g. broken singularities) and $U_{C}$ be the union of small enough neighborhoods around isolated-type singularities (e.g. cusps or Lefschetz singularities). We may take those open sets small enough such that $U_C\cap U_{\Gamma}=\emptyset$. For an isolated singularity $p\in M$, let $V_p$ be a neighborhood such that $V_p\subset \overline{V_p}\subset U_C$, and set $V_C$ the union of such open sets, over all isolated singularities. Analogously there exists an open set $V_{\Gamma}$ such that $V_{\Gamma}\subset \overline{V_{\Gamma}}\subset U_{\Gamma}$, containing the circle singularities. Denote by $\pi_{\Gamma}$ and $\pi_C$ the corresponding Poisson bivectors, constructed as above. That is, the bivectors of the \FR structure induced by the local form of $F$ around a singularity. Outside ${\rm Sing}(f)$, there exists a Poisson bivector $\pi_F$ whose symplectic foliation is given by the (regular) fibres of the $F$ (see Proposition 2.4 from \cite{GSSV15}). It is defined on $W:=M\setminus \left( \overline{V_{\Gamma}\cup V_C}\right)$. Hence, tere exist two non-vanishing smooth functions $g, h$ with:
\[
\pi_{\Gamma}=g\cdot \pi_F\ \mbox{on $W\cap U_{\Gamma}$}, \quad \pi_{C}=h\cdot \pi_F \ \mbox{on $W\cap U_{C}$}.
\]
Choose a connected component of $\overline{W\cap U_C}$. Let $\sigma$ be a cutt-off function on $W\cap U_C$. Similarly, we may take a cutt-off function $\lambda$, defined on a chosen connected component of $W\cap U_C$. On each connected component:
\[
\sigma(p)=
\begin{cases}
1 & \mbox{if $p\notin U_C$},\\
0 & \mbox{if $p\notin W$}
\end{cases}, \quad\quad \lambda(p)=
\begin{cases}
1 & \mbox{if $p\notin U_{\Gamma}$}\\
0 & \mbox{if $p\notin W$}.
\end{cases}
\]
\medskip

Define the function $\tau$ on $\overline{W\cap (U_C\cup U_{\Gamma})}$:
\[
\tau(p)=
\begin{cases}
1 & \mbox{if $p\notin U_C\cup U_{\Gamma}$},\\
0 & \mbox{if $p\in U_C\cup U_{\Gamma}$}.
\end{cases}
\]
Additionally, these functions can be chosen so that $\sigma+\lambda+\tau=1$. Then the bivector: 
\begin{equation}
\label{biv: global}
\Pi=\left(g\cdot \sigma+h\cdot \lambda+\tau\right)\pi_F
\end{equation}
defines a global Poisson structure on $M$, whose symplectic foliation is foliated by the fibres of $F$. In these terms, $M$ can be decomposed as 
\begin{equation}
\label{GlobalDecomp}
M=W\cup U_{C}\cup U_{\Gamma}.
\end{equation}

\section{Thom class and its image in Poisson cohomology}


For a vector bundle $\mathcal{E}: E\to M$ of rank $r$ over a closed orientable manifold of dimension $n$,  {\em the complex of forms with compact support in the vertical direction} is given by:
\begin{equation*}
\Omega_{cv}^*(E):=\{\omega\in \Omega^n(E)|\ \mbox{$\forall$ compact $K\subseteq M$, $\mathcal{E}^{-1}(K)\cap{\rm Supp}(\omega)$ is compact}\}.
\end{equation*}

The corresponding cohomology  $H^*_{cv}(E)$, with respect to the differential of forms $d$, is called the {\em vertical cohomology of $E$}.
\newline

Let $\{(U_{\alpha}, \psi_{\alpha})\}_{\alpha\in I}$ be an oriented atlas on $E$. Take a local trivialization $(U_{\alpha}, \psi_{\alpha})$, with coordinates $x=(x_1, \dots, x_n)$; and $s=(s_1, \dots, s_r)$ the corresponding fiber coordinates on $E|_{U_{\alpha}}$. Define a map $\mathcal{E}_{*}: \Omega_{cv}^*(E)\to \Omega^{*-r}(M)$ given by:
\[
\mathcal{E}_*(\omega|_{\mathcal{E}^{-1}(U_{\alpha})})=\begin{cases}
0 & \mbox{if $\omega|_{\mathcal{E}^{-1}(U_{\alpha})}=\mathcal{E}^*(\omega)f(x, s)ds_{i_1}\cdots ds_{i_l}$, $l<r$}\\
\mathcal{E}^*(\omega)\int_{\mathbb{R}^k} f(x, s) \mathbf{ds}& \mbox{if $\omega|_{\mathcal{E}^{-1}(U_{\alpha})}=\mathcal{E}^*(\omega) f(x, s) \mathbf{ds}$}.
\end{cases}
\]
Here  $\mathbf{ds}$ denotes the product measure form $ds_1\cdots ds_r$, and $\mathcal{E}^*(\omega)$ the pullback under $\mathcal{E}$ of differential forms. This map induces an isomorphism
\[
\mathcal{F}:=\mathcal{E}_*: \Omega_{cv}^*(E)\to \HdR{*-r}(M)
\]
called the {\em Thom isomorphism}. The image of $1$ in $\HdR{0}(M)$ determines a top cohomology class $\Phi\in H^n_{cv}(E)$ called the {\em Thom class}. Indeed, the Thom isomorphism can be defined by:
\begin{equation}
\mathcal{F}^{-1}(\omega)=\mathcal{E}^*(\omega)\wedge \Phi. 
\end{equation}

Consider $\mathcal{E}: N_S\to S$ the normal bundle on $S\subset M$ a submanifold of dimension $r$, which is a vector bundle of rank $r$. Let $j: N_S\to M$ be the inclusion and  $j_{*}: H^{*}_{cv}(N_S)\to \HdR{*}(M)$ its extension by $0$. In this context, Poincaré duality establishes the isomorphism $(\HdR{r}(M))^{*}\simeq \HdR{n-r}(M)$. The {\em Poincaré dual} of $S$ is the unique cohomology class $[\eta_S]\in \HdR{n-r}(M)$ such that:
\[
\int_S i^*\omega=\int_M\omega\wedge\eta_S
\]
for any $\omega\in \HdR{r}(M)$, where $i: S\to M$ is the inclusion. Then it is known that the Thom class of $S$ is its Poincaré dual, in the sense that $j_*(\Phi)=\eta_S$. See Section 6 in \cite{BottTu82}.
\newline

\begin{prop}

Let $M$ be a closed oriented smooth manifold of dimension $4$ and $S$ a closed immersed surface of genus $g$. Then the Poincaré class $\eta_S$ of $S$ can be represented by:
\[
\bar{\eta}_S:=\bar{f}_S\ dx^1\wedge dx^2.
\]
where 
\[
{\bar{f}_S}^2=
\begin{cases}
\frac{{\rm Vol}(S)}{{\rm Vol}(M)} & \mbox{in $S$},\\
0 & \mbox{in $M\setminus S$}.
\end{cases}
\]
\end{prop}

\begin{proof}
Recall that the de Rham Cohomology of $S$ is given by:
\[
\HdR{p}(S)=\begin{cases}
\mathbb{R} & \mbox{if $p=0, 2$}\\
\mathbb{R}^{2g} & \mbox{if $p=1$}.
\end{cases}
\]
\medskip

Observe then that for any $[\beta]\in\HdR{2}(M)$, $[i^*(\beta)]\in\HdR{2}(S)$, and then $[{\rm vol}_S]=[i^*(\beta)]$, where ${\rm vol}_S$ is the volume form of $S$. Furthermore, by definition of the Poincaré dual we have:
\begin{equation*}
{\rm Vol}(S)=\int_Si^*(\beta)=\int_M\beta\wedge\eta_S
\end{equation*}
for any $\beta\in\Omega^2(M)$. By taking a Riemannian metric we may have a $\star$-Hodge operator. Then, in the previous calculation we may take $\beta=\star\eta_S$. Write $\eta_S=f_S\ dx^1\wedge dx^2$, for $f_S=f(x^1, x^2, x^3, t)$ some smooth function on $M$. Then:
\begin{equation*}
\int_M\beta\wedge\eta_S=\int_M\star\eta_S\wedge\eta_S=\int_M f_S^2 \mu.
\end{equation*}

Then there exists a bump function $\bar{f}_S$ such that
\[
{\bar{f}_S}^2=
\begin{cases}
\frac{{\rm Vol}(S)}{{\rm Vol}(M)} & \mbox{in $S$},\\
0 & \mbox{in $M\setminus S$},
\end{cases}
\]
and that ${\bar{f}_S}^2=f^2$ a.e.
\end{proof}
\medskip

Let $\pi$ a Poisson bivector on $M$, and suppose that $S$ is a symplectic leaf of dimension $r$. Then we may apply the Thom isomorphism at $N_S$, which gives rise the sequence:
\[
\HdR{*}(S)\xrightarrow{\wedge \Phi} H^{*+n-r}_{cv}(N_S)\xrightarrow{j_{*}}\HdR{*+n-r}(M)\xrightarrow[\ ]{\bpi}\Hpi{*+n-r}(M).
\]
We may take $\bar{\eta}_S$ as the Poincaré dual of $S$. Thus for any differential form $\omega$ on $M$:
\[
(\bpi\circ j_*)(\omega\wedge \Phi)=\bar{f}_S\bpi(\omega)\wedge \bpi(dx^1)\wedge \bpi(dx^2).
\]

\subsection{Lie algebroid Thom class}

M. J. Pflaum, H. Posthuma and X. Tang defined a Thom class for Lie algebroids (see Sections 2.2 and 2.3 from \cite{PPT14}), determined by analogous properties as those of the de Rham Thom class. 
\newline

Recall that a {\em Lie algebroid} over a smooth manifold is a triple $(A, [, ], \sharp )$ of a vector bundle $A\to M$ over $M$, endowed with a bundle map $\sharp A\to TM$ called {\em anchor map}, and a Lie bracket $[, ]$ on the space of sections of $A$ that satisfies the Leibniz rule:
\[
[\alpha, f\beta]=\sharp(\alpha(f))\beta+f[\alpha, \beta].
\]

\example The cotangent bundle $T^*M$ of a Poisson manifold $(M, \pi)$ has a natural structure of Lie algebroid with its anchor map $\bpi:T^*M\to TM$ and the Lie bracket given by:
\[
[\alpha, \beta]=d(\pi(\alpha, \beta)) + \iota_{\bpi(\alpha)}d\beta-\iota_{\bpi(\beta)}d\alpha.
\]

\example The tangent bundle $TM$ with the identity as anchor map and the Lie bracket of vector fields give a natural structure of Lie algebroid to $TM$. It is called the {\em tangent algebroid}.
\newline

Given a Lie algebroid $(A, [, ], \sharp)$ there exists a differential complex with differential operator $d_A$ given by Cartan's formula, whose resulting cohomology is called the {\em algebroid cohomology}. If $A=TM$ is the tangent algebroid, then its cohomology is the de Rham cohomology of $M$. For the cotangent bundle one recovers the Poisson cohomology. 
\newline

Let $\mathcal{E}: E\to M$ be a vector bundle. The {\em pullback Lie algebroid along} $\mathcal{E}$ is a Lie algebroid  $\mathcal{E}^{!}(T^*M)$ over $TE$, given fiberwise at a point $m\in E$ by
\[
\mathcal{E}^{!}(T^*M)_m=\{(\alpha, \xi)\in T_{f(m)}^*M\oplus T_m(E)|\ \bpi(\alpha)=d\mathcal{E}(\xi)\},
\]
with anchor map $\rho_{\mathcal{E}^{!}}$, given by $\rho_{\mathcal{E}^{!}}(\alpha, \xi)=\xi$. Alternatively, the pullback Lie algebroid along $\mathcal{E}$ can be defined through a universal property:
\[
\xymatrix{\mathcal{E}^{!}\ar [d]_{\rho_{\mathcal{E}^{!}}} (T^*M) \ar[r]^{\rm proj}                                    &     T^{*}M \ar [d]^\rho \\
            TE \ar[r]_{d\mathcal{E}}   &    TM}
\]
where ${\rm proj}$ is the canonical projection over $T^*M$. See Section 4.2 in \cite{Mackenzie05} for more details.
\newline

The anchor map $\rho_{\mathcal{E}^{!}}$ induces a map at Lie algebroid cohomology level: 
\[
\rho^*_{\mathcal{E}^{!}}: H^*(T^*M, E)\to H^{*}(\mathcal{E}^{!}(T^*M), \mathcal{E}^*(E)).
\]

\begin{dfn}
If $\Phi$ is the Thom class of $E$, the {\em Lie algebroid Thom class} is defined by:
\[
{\rm Th}_{\pi}(E):=\rho^*_{\mathcal{E}^{!}}\Phi.
\]
\end{dfn}

\begin{prop}
Let $S$ be a symplectic leaf of dimension $r>0$, of a Poisson manifold $(M, \pi)$ then the pullback Lie algebroid of $T^*S$, along the normal bundle $\mathcal{E}:N_S\to S$ is:
\[
\mathcal{E}^{!}(T^*S)=T^*N_S. 
\]
whith anchor map $\rho_{\mathcal{E}^{!}}=\bpi$.
\end{prop}
\begin{proof}
Observe that the induced Lie algebroid on $S$ is given by the restriction to the Poisson structure at $S$. Note also that $S$ can be identified as the zero section of $N_S$, $Z_o: S\to N_S$. Then there exists a complementary space $\nu$, corresponding to $Z_o$ such that:
\[
Z_o^*(T^*N_S)=T^*S\oplus \nu.
\]
In particular, there exists a projection ${\bf p}: Z_o^*(T^*N_S)\to T^*S$ such that the following diagram commutes:
\[
\xymatrix{T^*\ar [d]_{\rho_{\mathcal{E}^{!}}}N_S \ar[r]^{\bf p}                                    &     T^*S \ar [d]^\bpi \\
            TN_S \ar[r]_{d\mathcal{E}}   &    TS}
\]
This also shows that $T^*N_S$ satisfies the universal property of pullback Lie algebroids. 
\end{proof}

\begin{teo}
The Lie algebroid Thom class of the normal bundle of a symplectic leaf $S$ of a Poisson manifold $(M, \pi)$ is
\[
{\rm Th}_{\pi}(N_S)=\bpi(\eta_S).
\]
\end{teo}
\begin{proof}
Notice that $\rho^*_{\mathcal{E}^{!}}$ is the induced map of the Lie algebroid cohomology of $TN_S$ and $T^*N_S$. Let $\omega\in H_{cv}^*(N_S)$ be an arbitrary form in the vertical cohomology of $N_S$. Denote by $\hat{j_{*}}$ the extension by $0$, $\hat{j_{*}}: \Hpi{*}(N_S)\to \Hpi{*}(M)$. Then, we immediatly have:
\[
(\hat{j_{*}}\circ\bpi)(\omega)=\bpi\circ j_{*}(\omega),
\]
which applied to the Thom class, gives the result.
\end{proof}

%
%
%

%
%

\section{Poisson cohomology classes of \FR structures}

The Poisson cohomology groups $\Hpi{k}(M)$ are $\Cas{M}$-modules. The group $\Hpi{0}(M)$ are the Casimirs of the Poisson structure. By means of the Thom isomorphism we may obtain cohomology classes in $\Hpi{2}(M)$ and $\Hpi{4}(M)$. The classes in each cases are obtained by cup product with the de Rham Thom class of the inclusion of the symplectic leaf in $M$. More precisely,

\begin{itemize}

\item $\Hpi{0}(M)={\rm Cas}_{\pi}(M)$.
\medskip

\item For $k=2$ or $k=4$, we have the sequence of mappings:
\[
\HdR{k-2}(S) \xrightarrow[]{\mathcal{F}^{-1}} \HdR{k}(N_S)\xrightarrow[]{j_{*}} \HdR{k}(M) \xrightarrow[]{\bpi} \Hpi{k}(M).
\]
If $k=2$, then it sends the generator of $\HdR{0}(S)$ to the Poisson class of the bivector
\[
\bpi(\eta_S)=\bar{f}_S\bpi(dx^1)\wedge \bpi(dx^2).
\]
%
For $k=4$, 
\[
\bpi({\rm vol}_S)=\bar{f}_S\bpi(dx^1)\wedge \bpi(dx^2)\wedge \bpi({\rm vol}_S).
\]
\end{itemize}
\medskip

\subsection{Group $\Hpi{3}(M)$} {\em Action of the monodromy of a singular fibration on the Poisson cohomology.}
\label{sec: PclassMon}
\medskip

Let $F:M\to \Sigma$ be a singular fibration and $G$ a group acting fiberwise and freely by diffeomorpisms outside its singular locus. Then it induces a homomorphism
\[
\pi_1(\Sigma\setminus {\rm Sing}(f))\to G,
\]
around a singularity of $F$, that captures the local behavior of the action, and the fibration itself. Then, it encloses the action of $G$ at level of Poisson cohomolgy. More precisely,  let $\Sigma_g$ be a surface of genus $g$, being the general fiber. By the homothopy lifting property, there exists a homomorphism:
\begin{equation}
\label{map: MonRepCohom}
\rho_0:{G}\to {\rm Aut}\left(\HdR{1}(\Sigma_g)\right)
\end{equation}
Assuming that $\rho_0$ preserves the intersection form in $\HdR{1}(M)$, then there exists a homomorphism
\[
\rho_G:\HdR{1}(\Sigma_g)\to \HdR{3}(M)
\]
given by the $G$-action followed by the cup product with the Thom Class of the inclusion $\Sigma_g\hookrightarrow M$. If $M$ is a Poisson manifold, then $\Sigma_g$ can be regarded as a symplectic leaf. We define: 

\begin{dfn}
Let $(M, \pi)$ be a Poisson manifold, whose singular leaves are the same as of a singular fibration. At the symplectic leaf $S$ of a singularity $p$ we define the homomorphism:
\[
{\rm Mon}_{\pi}:=\bpi\circ\rho_G: \HdR{1}(S)\to \Hpi{3}(M).
\]
where $\bpi$ is the Lichnerowicz homomorphism at level $3$.
\end{dfn}
\medskip

The objective of this section is to describe ${\rm Mon}_{\pi}$ for Lefschetz and wrinkled fibrations, as well as for each Lekili's move. 
\newline

For a given surface $\Sigma_g$ of genus $g$, ${\rm Diff}^{+}(\Sigma_g)$ will denote the group of orientation preserving diffeomorphisms of $\Sigma_g$, while ${\rm Diff}_0^{+}(\Sigma_g)$ is the subgroup of ${\rm Diff}^{+}(\Sigma_g)$ consisting of the diffeomorphisms isotopic to the identity. The {\em mapping class group} of $\Sigma_g$ is the quotient group 
\[
\mathcal{M}(\Sigma_g):={\rm Diff}^{+}(\Sigma_g)/ {\rm Diff}_0^{+}(\Sigma_g).
\]
Its group structure is given by concatenation of paths. For brevity, we will offten write $\mathcal{M}_g$.
\newline

Let $\gamma$ be a simple closed curve in $\Sigma_g$. A curve $\gamma$ is {\em separating} if $\Sigma_g\setminus \gamma$ is disconnected, otherwise it is {\em nonseparating}. A nonseparating curve is cohomologically non-trivial. For any tubular neighborhood of $\gamma$, there exists a diffeomorphism $\psi: {\rm Tub}(\gamma)\to S^1\times [0, 1]$. Given a map $d: S^1\times [0, 1]\to S^{1}\times[0, 1], d(\theta, t)=(\theta+2\pi t, t)$, a {\em Dehn twist} $D_{\gamma}:\Sigma_g\to \Sigma_g$ along $\gamma$ is a homeomorphism defined as 
\[
D_{\gamma}:=\begin{cases}
\psi^{-1}\circ d\circ \psi & \mbox{in $ {\rm Tub}(\gamma)$,}\\
{\rm Id} & \mbox{in $\Sigma_g\setminus  {\rm Tub}(\gamma)$}.
\end{cases}
\]
\medskip

The well known Dehn-Lickorish theorem states that the mapping class group ${\rm Mod}(\Sigma_g)$ is generated by Dehn twists around $3g-1$ nonseparating simple curves. S. Humphries found $2g+1$ generators \cite{Humphries79}. 
\newline

 Let $f:M\to \Sigma$ be a Lefschetz fibration with genus $g$ surfaces as fibres. Then its singular locus is a finite set:
\[
{\rm Sing}(f)=\{b_1, \dots, b_r\}.
\]

For a fixed point $x_o\in {\rm Int}(\Sigma)\setminus {\rm Sing}(f)$, attached to the general fibre $\Sigma_g$, consider a loop $\gamma:[0, 1]\to Y\setminus {\rm Sing}(f)$ at $x_o$. Idenfity $f^{-1}(x_o)$ with $\Sigma_g$ by an orientation preserving diffeomorphism $\Phi:\Sigma_g\to f^{-1}(x_o)$. There exists a diffeomorphism 
\[
\varphi: [0, 1]\times \Sigma_g\to X\setminus f^{-1}({\rm Sing}(f))
\] 
that preserves orientation with $\varphi(0, p)=\Sigma_p$, and $f(\varphi(t, p))=\gamma(t).$
\medskip

\begin{dfn}
The monodromy of $\gamma$ associated with $\Phi$ is the isotopy class of $\Phi^{-1}\circ \varphi(\cdot, 1): \Sigma_g\to \Sigma_g$. The group homomorphism
\[
\pi_1(\Sigma\setminus {\rm Sing}(f))\to \mathcal{M}_g, \quad \gamma\to [\Phi^{-1}\circ\varphi(\cdot, 1)]
\]
is called {\em the monodromy representation}. It is well defined up to conjugacy by $\mathcal{M}_g$.
\end{dfn}

Y. Matsumoto showed that any two Lefschetz fibrations are isomorphic if and only if their monodromy representations are equivalent \cite{Mats96}. 
\newline

Let $f:M\to \Sigma$ a BLF, suppose that it has exactly one singular circle of indefinite folds and one Lefschetz singularity $p_s$. Let $\Gamma$ be the singular circle and  $p_f\in \Gamma$. Let $c$ be a path that connects a fixed regular point $p$ to $p_f$; and $\gamma$ a disjoint path to $c$ that connects $p$ to $p_s$. Then the tuple $(c; \gamma)$ determines completely the BLF if and only if the Dehn twist along $\gamma$, $D_{\gamma}$, belongs to the subgroup of $\mathcal{M}_g$
\[
\left\{D\in \mathcal{M}_g|\ D:\Sigma_g\to \Sigma_g, D(c)=c \right\}.
\]

More generally,  given a BLF, there exists a tuple $(c; \gamma_1, \dots, \gamma_{l})$ of simple closed non-separating curves on $\Sigma_g$ satisfying that the product of Dehn twists $D_{\gamma_{1}}\cdots D_{\gamma_{l}}$ lies into the subgroup $\mathcal{M}_g[c]$ of $\mathcal{M}_g$, formed by diffeomorphisms that fix $c$. Such a tuple is called  a {\em Hurwitz system}. Two BLF's are isomorphic if and only if their Hurwitz systems are equivalent \cite{BayHay15}. Even more, there exists a map 
\begin{equation*}
\Phi_{c}:\mathcal{M}_g\to \mathcal{M}_{g-1}
\]
such that $D_{\gamma_{1}}\cdots \gamma_{l}\in \ker(\Phi_c)$, and that factors by:
\[
\mathcal{M}_g\to \mathcal{M}_g(\Sigma_g\setminus N_c), \mbox{\ and } \mathcal{M}_g(\Sigma_g\setminus N_c)\to \mathcal{M}_{g-1},
\]
where $N_{c}$ is a tubular neighborhood of $c$. Recall that we are assuming that there is only one singular circle, then $N_c$ contains no other singular circles. See \cite{Baykur09}. Note that the mapping and its factors describe a surgery process that is performed when crossing a singular circle where the genus of the general fiber drops down by 1. The kernel of $\Phi_c$ is generated by lifts of point pushing maps and $D_c$ (see Lemma 3.1 of \cite{BayHay15}). Then the equation:
\[
D_{1}\cdots D_l=\pm c
\]
determines the monodromy representation. One says that the monodromy representation is contained in a subgroup $H< \mathcal{M}_g$ if all $D_1, \dots, D_l$ belong to $H$, up to conjugacy. 
\newline

\begin{lemma}
The subgroup  $\mathcal{M}_g[c]$ of $\mathcal{M}_g$, formed by diffeomorphisms that fix $c$, acts freely on the closed surface $\Sigma_{g-1}$, which is obtained from the general fiber $\Sigma_g$ by removing a handle along $c$, and gluing in two disks. Then we have a homomorphism:
\[
\rho_0:\mathcal{M}_g[c]\to {\rm Aut}\left(\HdR{1}(\Sigma_g)/\langle [c] \rangle \right)\simeq {\rm Aut}\left(\HdR{1}(\Sigma_{g-1})\right).
\]
It is the induced map from the monodromy representation:
\[
\pi_1(\Sigma\setminus \left({\rm Sing}(f)\cup N\right))\to \mathcal{M}_{g-1}
\]
\end{lemma}
\begin{proof}
The group $\HdR{1}(\Sigma_{g-1})$ has $2g-2$ generators, given by pairs of transversal curves $(a_i, b_i)$, $i=1, \dots, g-1$. From $\Sigma_{g-1}$ we remove two disks with centers $p_1, p_2$. Furthermore, we have a natural identification $\Sigma_{g-1}\setminus \{p_1, p_2\}$ with the quotient $\Sigma_g/c$. Then we glue a handle, generated by $c$. The resulting surface adds two additional transversal curves $(a_c, b_{1, 2})$. The curve $a_c$ identified with $c$; while $b_{1, 2}$ with the  curve that connects $p_1, p_2$ along the attached handle. This shows that $\HdR{1}(\Sigma_{g-1})\simeq \HdR{1}(\Sigma_{g}/\langle [c] \rangle$. The result follows. See Section also 3.18 in \cite{FarbMar12}.
\end{proof}
\medskip
%

\begin{prop}
Let $f: M\to \Sigma$ be a BLF, and denote by $\FRb{M}$ its associated \FR structure with bivector $\pi$. Then we have:
\begin{enumerate}
\item[i)] In a neighborhood of a Lefschetz singularity, it is given by the monodromy representation, ${\rm Mon}_{\pi}=\bpi\circ \rho_{\mathcal{M}_g}$, and
\medskip

\item[ii)] in a tubular neighborhood of a singular circle, it is given by ${\rm Mon}_{\pi}=\bpi\circ \rho_{\mathcal{M}_g[c]}$.
\end{enumerate}
\end{prop}
\medskip

Observe that around a broken singularity, if the general fiber is a surface of genus $g$, the monodromy is contained in the mapping class group of $\Sigma_{g-1}$. 



%
%
%
%

\begin{teo}
Let $M$ be a \FR manifold with generic singularities. Then the homomorphism ${\rm Mon}_{\pi}$ at a symplectic leaf is determined by the action of Dehn twists on $\HdR{1}(\Sigma_g)$ or Dehn twists on $\HdR{1}(\Sigma_{g-1})$. It gives non-trivial classes in $\Hpi{3}(M)$. 
\end{teo}

\subsection{Lekili's moves and their monodromies}

Summarizing, a wrinkling fibration may have cusps, indefinite folds or broken singularities as critical points. The effect of the monodromy in the Poisson manifold around indefinite folds is determined by a homomorphism $\Phi_{c}:\mathcal{M}_g\to \mathcal{M}_{g-1}$ which induces a homomorphism $\tilde{\Phi}_{c}: \HdR{1}(\Sigma_{g})\to \HdR{1}(\Sigma_{g-1})$, for a nonseparating curve $c$, and then by composing with the Lichnerowicz homorphism. At broken singularities, the effect is given by Dehn twists along vanishing cycles. While in the case of a cusp, when one approaches to a cusp, the general fiber increases its genus by 1. In the reverse process, the Lefschetz singularity is replaced by three cusps.

\subsubsection*{Birth}
The only values of the parameter $s$ at which a birth mapping produces singularities are when $s=0$ or $s>0$. In the first case, there is only one singularity at origin, which is a cusp. In fact, as was described by Lekili, this move substitutes this cusp singularity by a Lefschetz singularity, said otherwise, one obtains a Morse function. 

For the second case, the singular locus is the circle $\{x^2+t^2=s, y=z=0\}$, obtained by gluing two cusps. The critical value set is the union of the lines $\{x=1, y=z=0\}$ and $\{x=-1, y=z=0\}$. Let $L$ be a segment joinging these lines. Then, along $L$ the fiber degenerates by increasing its genus by 1. Thus, when a birth of a cusp singularity occurs, there exists a homomorphism
\begin{equation*}
\hat{\Phi}_{c}:\mathcal{M}_{g}\to \mathcal{M}_{g+1}
\]
describing the attaching of a handle. Analogously as before, one has a homomorphism $\HdR{1}(\Sigma_{g})\to \HdR{1}(\Sigma_{g+1})$, by the homotopy invariance property of the de Rham cohomology. It determines the homomorphism ${\rm Mon}_{\pi}$ for this case. 

\subsubsection*{Merging.}
This move describes the gluing of two singular circles. Recall that the monodromy around a singular circle is given by a homomorphism $\Phi_c$. In fact, if $s>0$ we have a wrinkled map whose singular locus is the gluing of two cusps. For $s>0$ we have a cusp singularity. For $s=0$ is similar as in the previous one.

	
\subsubsection*{Flipping.}
For $s<0$ the singular locus is a simple curve (with no cusps). Along the singular locus the genus of the general fiber increases by 1. For $s=0$ one have a higher order singularity, with a similar behavior. For $s>0$, a flipping behavior happens. Let $a$ be a separating curve at the general fiber. Then, the monodromy is given by a Dehn twist around a nonseparating curve. The resulting map ${\Phi}_{a}:\mathcal{M}_{g}\to \mathcal{M}_{g}$ preserves the genus of the fiber, but it factors via removing tubular neighborhood of a nonseparating curve $a$ and a pushing map along a curve $b$, regarding $a$ and $b$ as generators in cohomology.

\subsubsection*{Wrinkling. }
In this case, for each parameter $s$, the singular locus is given by $\{(t, x, y, z)|\ x^2+t^2+st=0, y=z=0\}$. It is a curve with 3 cusps. 
\newline

We also refer the reader to \cite{Hayano14}, where the K. Hayano described the change of monodromy for wrinkled fibrations, in terms of vanishing cycles. 
\medskip

\subsection{The modular class}

The {\em modular class} $[Z_{\pi}]$ of an oriented Poisson manifold $\FRb{M}$ is the Poisson cohomology class in $\Hpi{1}(M)$ represented by the (global) vector field $Z_{\pi}$ defined by:
\[
Z_{\pi}(h):={\rm div}^{\mu}(\bpi(dh))
\]
where ${\rm div}^{\mu}$ is the divergence with respect to $\mu$. The vector field $Z_{\pi}(f)$ is called the {\em modular vector field}. If $Z_{\pi}(f)$ vanishes everywhere, then $\pi$ is called {\em unimodular.} \FR structures are unimodular \cite{Prz01}.


\begin{lemma}
Consider a Poisson bivector on $\mathbb{R}^4$ given by:
\[
\pi(p)=k\cdot \left[A_1 \frac{\partial }{\partial y}\wedge \frac{\partial }{\partial z}+A_2 \frac{\partial }{\partial z}\wedge \frac{\partial }{\partial x}+A_3 \frac{\partial }{\partial x}\wedge \frac{\partial }{\partial y}\right].
\]
with smooth functions $k, A_i: \mathbb{R}^4\to \mathbb{R}$, $i=1, \dots, 4$, being $k$ non-vanishing. Then the modular vector field with respect to the volume form is given by:
\begin{equation}
\label{exp: Modvec}
Z_{\pi}=\left\langle{\rm rot}\left[A_1, A_2, A_3\right], \left(\frac{\partial}{\partial x}, \frac{\partial}{\partial y}, \frac{\partial}{\partial z}\right)\right\rangle-\bpi(d\log(|k|)).
\end{equation}
Here ${\rm rot}$ and $\langle, \rangle$ denote the rotational operator and  the euclidean inner product in $\mathbb{R}^3$, respectively.
\end{lemma}
\begin{proof}
Consider $(x_1, x_2, x_3, x_4):=(t, x, y, z)$, and $\zeta_i:=\frac{\partial }{\partial x_i}$. Then if we represent the wedge product as a mutiplication on the variables $\zeta_i$, they commute with the variables $x_i$ but anti-commute among themselves. In these terms, it is known that in a local system of coordinates $(t, x, y, z)$ the modular class can also be written as:
\[
Z_{\pi}=\sum_{i=1}^4 \frac{\partial^2 \pi}{\partial x_i\partial \zeta_i}.
\]
See formula (2.89) in \cite{DuZu05}. The differentiation rule is:
\[
\frac{\partial \zeta_{i_1}\cdots \zeta_{i_p}}{\partial \zeta_{i_k}}=(-1)^{p-k}\zeta_{i_1}\cdots \hat{\zeta}_{i_k}\cdots \zeta_{i_p},
\]
where $\hat{\zeta}_{i_k}$ denotes that the term $\zeta_{i_k}$ is missing, for $1\leq k\leq p$. Then the Poisson bivector under consideration can be written locally as:
\[
\pi=k\cdot\left[A_1 \zeta_3\cdot\zeta_4+A_2\zeta_4\cdot \zeta_2+A_3 \zeta_2\cdot \zeta_3\right]
\]
Assume for a moment that $k\equiv 1$. Then we have:
\begin{eqnarray*}
\frac{\partial}{\partial x_1}\left(\frac{\partial\pi}{\partial \zeta_1}\right)&=&0\\
\frac{\partial}{\partial x_2}\left(\frac{\partial\pi}{\partial \zeta_2}\right)&=&\frac{\partial}{\partial x_2}\left(A_2\zeta_4-A_3\zeta_3\right)=\frac{\partial A_2}{\partial x_2}\zeta_4-\frac{\partial A_3}{\partial x_2}\zeta_3\\
\frac{\partial}{\partial x_3}\left(\frac{\partial\pi}{\partial \zeta_3}\right)&=&\frac{\partial}{\partial x_3}\left(A_3\zeta_2-A_1\zeta_4\right)=\frac{\partial A_3}{\partial x_3}\zeta_2-\frac{\partial A_1}{\partial x_3}\zeta_4\\
\frac{\partial}{\partial x_4}=\left(\frac{\partial\pi}{\partial \zeta_4}\right)&=&\frac{\partial}{\partial x_4}\left(A_1\zeta_3-A_2\zeta_2\right)=\frac{\partial A_1}{\partial x_4}\zeta_3-\frac{\partial A_2}{\partial x_4}\zeta_2.
\end{eqnarray*}
Therefore:
\[
Z_{\pi}=\left\langle{\rm rot}\left[A_1, A_2, A_3\right], \left(\frac{\partial}{\partial x}, \frac{\partial}{\partial y}, \frac{\partial}{\partial z}\right)\right\rangle.
\]
Using Proposition 2.6.5 in \cite{DuZu05}, by our choice of the volume form, we obtain the expression (\ref{exp: Modvec}). 

\end{proof}
\medskip

The lemma above also gives a local expression for the modular vector field for those Poisson bivectors considered by Pichereau \cite{Pichereau06}, in $\mathbb{R}^3$. It directly implies that those structures are unimodular.
\medskip

\begin{teo}
Any \FR manifold $\FRb{M}$ is unimodular. If it has generic singularities, its modular class is given locally by the modular vector field:
\[
Z_{\pi}=\left\langle{\rm rot}\left[\frac{\partial \psi}{\partial x}, -\frac{\partial \psi}{\partial y}, \frac{\partial \psi}{\partial z}\right], \left(\frac{\partial}{\partial x}, \frac{\partial}{\partial y}, \frac{\partial}{\partial z}\right)\right\rangle.
\]
\end{teo}
\medskip

\section{Examples and final remarks}

For a $\FR$singular structure, given by a singular mapping $f:M\to \Sigma$ (generic or not), the composite map ${\rm Mon}_{\pi}$ gives non-trivial Poisson cohomology classes, provided that $M$ has non-vanishing third de Rham cohomology group. There is a wide class of manifolds where this condition is fulfilled. For instance, it is known that a BLF exists in any homotopy class of applications from $4$-manifolds to the sphere. Thus it is enough if we require $M$ having non-vanishing third de Rham cohomology group. Nevertheless, it is worth mentioning explicit cases where the condition on the de Rham cohomology is satisfied. The following is a list of manifolds with such a condition.
\newline

\begin{itemize}
\item {\bf Genus-$1$ BLF's.} In terms of Hayano's classification theorem for genus-$1$ BLF's, it is known that the manifolds
\[
M= S^1\times S^3\# S\#l\overline{\mathbb{CP}^2}
\]
where $S$ is an $S^2$-bundle over $S^2$, admits a BLF with genus-$1$ fibers, $l$ Lefschetz critical points, and an indefinite fold. See \cite{Hayano14c}. 

\item {\bf Genus-$2$ BLF's.} R. I. Baykur and M. Kormaz proved that the manifolds
\[
(S^2\times T^2)\#m\overline{\mathbb{CP}^2}
\]
for $m=3, 4$ and $12$ admits genus-$2$ Lefschetz fibration i.e, with no broken singularities. See \cite{BaykurKormaz17} and \cite{Mats96}.
\medskip

\item {\bf Arbitrary genus-$g$ BLF's.} The manifolds:
\[
M= \Sigma_{(g-k+1)/2}\times S^2\#4k\overline{\mathbb{CP}^2},\ \mbox{$k=1$ if $g$ is even; $k=2$ if $g$ is odd}
\]
admit a genus-$g$ Lefschetz fibration over $S^2$. See \cite{EndoNagami05}.
\medskip

\item {\bf BLF on products.} It was exhibited that for any closed, oriented smooth $3$-manifold $M$, the product $M\times S^1$ admits a BLF. Then for the purpose of ${\rm Mon}_{\pi}$, we may take $M$ being connected.
\end{itemize}

Recall that one may introduce a wrinkled structure at a Lefschetz singularity. Then all previous cases are also examples of total spaces of wrinkled fibrations.

\subsubsection*{Global cohomology classes}

The Poisson cohomology version of the Mayer-Vietoris sequence establishes that for every open subsets $(U, V)$ in a Poisson manifold $M$, the sequence:
\[
\cdots\to \Hpi{*}(U\cup V)\to \Hpi{*}(U)\oplus \Hpi{*}(V)\to \Hpi{*}(U\cap V)\to \Hpi{*+1}(U\cup V)\to\cdots
\]
is exact. Then in terms of the global decomposition of a \FR manifold \ref{GlobalDecomp}
\begin{equation*}
M=W\cup U_{C}\cup U_{\Gamma},
\end{equation*}
where $W\cap U_{C}\neq \emptyset$, $W\cap U_{\Gamma}\neq \emptyset$ and $U_{C}\cap U_{\Gamma}=\emptyset$. Thus one obtains that the Poisson cohomology splits as
\[
\Hpi{*}(M)=\HdR{*}^c(W\setminus (U_{C}\cup U_{\Gamma}))\oplus\Hpi{*}(U_{C})\oplus \Hpi{*}(U_{\Gamma}).
\]
Here we denote by $\HdR{*}^c$ the complactly supported de Rham cohomology.

\bibliographystyle{plainurl}

\bibliography{bibthom}
%

\end{document}